\documentclass[num-refs]{article}

\usepackage{amsthm}
\usepackage[letterpaper, total={6in, 8in}]{geometry}
\usepackage{graphicx}
\usepackage{booktabs,array,multirow}
\usepackage{amsfonts,amsmath,amssymb}
\usepackage{hyperref}
\hypersetup{colorlinks=false,pdfborder={0 0 0}}

\usepackage[mathscr]{euscript}
\usepackage{multicol} 
\usepackage{tikz}
\usepackage{tikz-cd}
\usepackage{mathtools}
\usepackage{marvosym}
\usepackage{hyperref}
\usepackage{cleveref}
\usepackage{nicefrac}
\usepackage{thm-restate}
\usepackage{caption}
\usepackage[labelfont={}]{subcaption}
\usepackage{xspace}
\usepackage{url}
\newtheorem{theorem}{Theorem}
\newtheorem*{theorem*}{Theorem}
\newtheorem{lemma}[theorem]{Lemma}
\newtheorem{proposition}[theorem]{Proposition}
\newtheorem{corollary}[theorem]{Corollary}

\theoremstyle{definition}
\newtheorem*{definition*}{Definition}

\theoremstyle{remark}
\newtheorem{remark}[theorem]{Remark} 
\numberwithin{equation}{section}

\crefname{lemma}{Lemma}{Lemmas}
\crefname{figure}{Figure}{Figures}
\crefname{theorem}{Theorem}{Theorems}
\crefname{equation}{Equation}{Equations}
\crefname{corollary}{Corollary}{Corollaries}
\crefname{table}{Table}{Tables}
\crefname{section}{Section}{Sections}
\crefname{remark}{Remark}{Remarks}
\crefname{proposition}{Proposition}{Propositions}

\newcommand{\mysc}[1]{\textrm{\textsc{#1}}}

\graphicspath{{../figures/}}
\newcommand{\HH}{\mathbb{H}}

\newcommand{\vertiii}[1]{{\left\vert\kern-0.25ex\left\vert\kern-0.25ex\left\vert #1 
		\right\vert\kern-0.25ex\right\vert\kern-0.25ex\right\vert}}

\DeclarePairedDelimiter\ceil{\lceil}{\rceil}
\DeclarePairedDelimiter\floor{\lfloor}{\rfloor}

\newcommand*\tri{\vcenter{\hbox{\includegraphics[width=.7em]{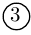}}}}
\DeclareMathOperator{\crg}{cr}
\DeclareMathOperator{\bcr}{bcr}
\DeclareMathOperator{\ctcr}{cr_{\tri}}

\newcommand*\circled[1]{\tikz[baseline=(char.base)]{
		\node[shape=circle,draw,inner sep=.5pt](char){\fontsize{4pt}{0cm}\selectfont #1};}}
\newcommand{\crN}[1]{\crg_{\circled{$#1$}}}
\newcommand{\bcrN}[1]{\bcr_{\circled{$#1$}}}

\newcommand{\vtx}[4]{#1_{#2}(\mysc{#3,#4})}

\newcommand{\mind}{\min d\xspace}
\newif\iflatexml\latexmlfalse

\AtBeginDocument{\DeclareGraphicsExtensions{.pdf,.PDF,.eps,.EPS,.png,.PNG,.tif,.TIF,.jpg,.JPG,.jpeg,.JPEG}}

\usepackage[utf8]{inputenc}
\usepackage[english]{babel}

\usepackage{siunitx}
\usepackage{authblk}

\title{Bounding the tripartite-circle crossing number of complete tripartite graphs}

\author[1]{Charles Camacho}
\affil[1]{University of Washington}

\author[2]{Silvia Fern{\'a}ndez-Merchant}
\affil[2]{California State University, Northridge}

\author[3]{Marija Jeli{\'c} Milutinovi{\'c}}
\affil[3]{University of Belgrade}

\author[4]{Rachel Kirsch}
\affil[4]{George Mason University}

\author[5]{Linda Kleist}
\affil[5]{Technische Universit{\"a}t Braunschweig}

\author[6]{Elizabeth B. Matson}
\affil[6]{Alfred University}

\author[7]{Jennifer White}
\affil[7]{Saint Vincent College}

\begin{document}
	
	\maketitle
	\selectlanguage{english}
	\begin{abstract}
		A tripartite-circle drawing of a tripartite graph is a drawing in the plane, where each part of a vertex partition is placed on one of three disjoint circles, and the edges do not cross the circles. 
		We present upper and lower bounds on the minimum number of crossings in  tripartite-circle drawings of $K_{m,n,p}$.
		In contrast to 1- and 2-circle drawings, which may attain the Harary-Hill bound, our results imply that balanced restricted 3-circle drawings of the complete graph are not optimal.
		
	\end{abstract}

	\section{Introduction}

The \emph{crossing number} of a graph $G$, denoted by $\crg(G)$, is the minimum number of edge crossings over all drawings of $G$ in the plane. It quantifies how close or far a graph is from being planar. Drawings with few crossings have been studied in connection with readability and VLSI chip design \cite{leighton}. See \cite{Sch14} for a survey of crossing number variants and some of their applications. 
Computing the crossing number of a graph is an NP-hard problem \cite{garey1983NP,HlinenyNP}. The precise values are not known even for very special graph classes such as complete and complete bipartite graphs. Nevertheless, there exist long-standing conjectures. Zarankiewicz~\cite{Z} conjectured that for the complete bipartite graph $K_{m,n}$, the bound
\[\crg(K_{m,n})\leq 
\bigg\lfloor\frac{m}{2}\bigg\rfloor\left\lfloor\frac{m-1}{2}\right\rfloor
\bigg\lfloor\frac{n}{2}\bigg\rfloor\left\lfloor\frac{n-1}{2}\right\rfloor \eqqcolon Z(m,n),\]
given by a certain straight-line drawing of $K_{m,n}$ with vertices placed along two axes, is the best possible. Later, Harary and Hill~\cite{HH} conjectured that the upper bound for the complete graph $K_n$,
\[
\crg(K_n)\leq  \frac{1}{4}\bigg\lfloor\frac{n}{2}\bigg\rfloor\left\lfloor\frac{n-1}{2}\right\rfloor\left\lfloor\frac{n-2}{2}\right\rfloor\left\lfloor\frac{n-3}{2}\right\rfloor \eqqcolon H(n),
\]
given by Guy~\cite{G1960}, is the best possible.

Among the best known drawings of complete graphs are drawings where the vertices are placed on one or two circles and edges do not cross the circles. Such drawings are $1$-circle drawings (or \emph{$2$-page book drawings})~\cite{AAFRS13} and $2$-circle drawings (or \emph{cylindrical drawings})~\cite{AAFRS}, respectively. (For more details, refer to Section \ref{sec:HH}.) A question of interest \cite{JK, AAFRV} is to determine which other families of drawings of~$K_n$ achieve the conjectured minimum number of crossings, $H(n)$. One possible direction is to look at greater numbers of circles.

As a natural extension of $1$- and $2$-circle drawings,  a \emph{$k$-circle drawing} of a graph~$G$ in the plane is a drawing in which the vertices are placed on $k$ disjoint circles and the edges do not cross the circles~{\cite{FGHLM}}. The minimum number of crossings in a $k$-circle drawing of a graph~$G$ is the \emph{$k$-circle crossing number} of~$G$. For the special case when $G$ is a $k$-partite graph, if we further require that the vertices on each circle form an independent set, we call these drawings \emph{$k$-partite-circle drawings}. 
We call the minimum number of crossings in a $k$-partite-circle drawing the 
\emph{$k$-partite-circle crossing number} and denote it by $\crN{k}(G)$.
In this paper, we determine bounds for the tripartite-circle crossing number of complete tripartite graphs, and we conclude that for $n \ge 13$ there are no balanced restricted 3-circle drawings of $K_n$ that achieve the minimum number of crossings.

\subsection*{Previous results and related work}

In this section, we concentrate on bipartite-circle crossing numbers and on the crossing numbers of complete tripartite graphs.
The $2$-circle drawings are also called cylindrical because they can be thought of as drawings on the surface of a cylinder, with the vertices on the top and bottom circles. Analogously, a $3$-circle drawing can also be understood as a pair of pants drawing (an instance of the \emph{map crossing number} \cite{PSS}), where two circles are enclosed by the third. A \emph{radial drawing} \cite{B2007} of two concentric circles is equivalent to a 2-circle drawing (cylindrical drawing). For $k$-circle drawings with $k\geq 3$, three or more concentric circles {(or more generally three pairwise nested circles)} would require that any edges from the outermost to the innermost circle would necessarily cross the middle circle(s), so a radial drawing with three or more concentric circles is not equivalent to a $k$-circle drawing. {Consequently, 3-partite-circle drawings of complete tripartite graphs do not contain three pairwise nested circles.}

The bipartite-circle drawings of bipartite graphs, in which the vertices of each part are placed on a circle and no edge crosses a circle, are of special interest due to their connection to one of the conjectured optimal drawings of $K_n$. A $2$-circle drawing of $K_n$ with $H(n)$ crossings can be obtained from a bipartite-circle drawing of $K_{\floor{n/2},\ceil{n/2}}$ by adding straight-line edges between vertices on the same circle. In general, the bipartite-circle crossing number of complete bipartite graphs, also known as the \emph{bipartite cylindrical crossing number}, is fully understood. In 1997, Richter and Thomassen~\cite{RT} settled the balanced case by showing that
\begin{equation}
	\crN{2}(K_{n,n})=n\binom{n}{3}.
	\label{eq:bipartitebalanced}
\end{equation} 
\'Abrego, Fern\'andez-Merchant, and Sparks~\cite{AFS} generalized this result to all complete bipartite graphs. For $m\leq n$, the bipartite-circle crossing number is
\begin{align}
	\crN{2}(K_{m,n})
	=\binom{n}{2}\binom{m}{2}
	+&\sum_{1\leq i<j \leq m} \left(\left\lfloor\frac{n}{m}(j-1)\right\rfloor-\left\lfloor\frac{n}{m}(i-1)\right\rfloor\right)^2 \nonumber \\ 
	-n\cdot&\sum_{1\leq i<j \leq m} \left(\left\lfloor\frac{n}{m}(j-1)\right\rfloor-\left\lfloor\frac{n}{m}(i-1)\right\rfloor\right).
	\label{eq:bipartiteGeneral}
\end{align}
In particular, if $m$ divides $n$, then 
$\crN{2}(K_{m,n})=\frac{1}{12}n(m-1)(2mn-3m-n).$

For the general crossing number of complete tripartite graphs, Gethner et al.~\cite{GHLPRY} proved an upper bound $A(m,n,p)$ on $\crg(K_{m,n,p})$ that is analogous to the Zaran\-kie\-wicz Conjecture for $K_{m,n}$. 
Additionally, they proved that among 
\emph{straight-line drawings}
their bound is asymptotically very close to best possible. For balanced tripartite graphs, $A(n,n,n)$ is of order $\nicefrac{9}{16}\cdot n^4$, much less than our lower bound of $\nicefrac{5}{4}\cdot n^4$ (see \cref{th:balanced}) because their drawings are not restricted by circles. (A similar gap exists between $Z(n,n) \sim \nicefrac{1}{16}\cdot n^4$ and $\crN{2}(K_{n,n}) \sim \nicefrac{1}{6}\cdot n^4$.) Asano~\cite{A1986} determined the crossing numbers of $K_{1,3,n}$ and $K_{2,3,n}$, and Ginn and Miller~\cite{GM} gave bounds on $\crg(K_{3,3,n})$. 
{
	More recently, building upon this work, we~\cite{K22n} established the exact tripartite-circle crossing number of $K_{2,2,n}$ for every integer $n\geq 3$, as}
\begin{equation}
	\crN{3} (K_{2,2,n})=
	6\left\lfloor\frac{n}{2}\right\rfloor\left\lfloor\frac{n-1}{2}\right\rfloor+2n-3.
	\label{eq:K22n}
\end{equation}
For other crossing number results and equivalent terminology, see e.g.~\cite{Sch14}.

\subsection*{Our results.}
We prove several bounds on the tripartite-circle crossing number of complete tripartite graphs. 
\begin{theorem}\label{th:general}
	Let $m$, $n$, and $p$ be natural numbers and $\mathfrak t:=\{(m,n,p),(n,p,m),$ $(p,m,n)\}$. Then the following bounds hold:
	\begin{equation*}
		\sum_{(a,b,c)\in\mathfrak t}
		\left(
		\crN{2}(K_{a,b})
		+ \ ab
		\bigg\lfloor \frac{c-1}{2}\bigg\rfloor \bigg\lfloor \frac{c}{2}\bigg\rfloor
		\right) 
		\leq \crN{3}(K_{m,n,p})
		\leq
		\sum_{(a,b,c)\in\mathfrak t}
		\left(\binom{a}{2}\binom{b}{2}
		+ \ ab
		\bigg\lfloor \frac{c-1}{2}\bigg\rfloor\bigg\lfloor \frac{c}{2}\bigg\rfloor
		\right).
	\end{equation*}
\end{theorem}
For $m,n,p \ge 3$ we improve the lower bound by 2 in \cref{cor:improvedLower}.  Using \cref{eq:bipartitebalanced} and \cref{cor:improvedLower}, \cref{th:general} simplifies as follows for the balanced case.
Note that the lower bound of order $\sim \nicefrac{5}{4}\cdot n^4$ and the upper bound of order $\sim \nicefrac{6}{4}\cdot n^4$ are fairly close.
\begin{corollary}\label{th:balanced}
	For any integer $n \geq 3$,
	\[3n\binom{n}{3}+3n^2\bigg\lfloor \frac{n}{2}\bigg\rfloor\left\lfloor \frac{n-1}{2}\right\rfloor +2
	\leq \crN{3}(K_{n,n,n})\leq
	3\binom{n}{2}^2+3n^2\bigg\lfloor \frac{n}{2}\bigg\rfloor\left\lfloor \frac{n-1}{2}\right\rfloor.\]
\end{corollary}

Finally, $k$-partite-circle drawings of complete $k$-partite graphs easily give rise to drawings of complete graphs, known as \emph{restricted $k$-circle drawings}, by adding straight-line segments between each pair of vertices on the same circle. If the numbers of vertices on the circles are as close to equal as possible, these drawings are called \emph{balanced restricted $k$-circle drawings} of~$K_n$. The minimum number of crossings in such a drawing is denoted by $\bcrN{k}(K_n)$. Certain balanced $1$- and $2$-circle drawings of the complete graph have $H(n)$ crossings and are conjectured to be optimal 
\cite{BW,BK,AAFRS13,HH,AAFRS}. Our results imply that this phenomenon does not generalize to balanced restricted $3$-circle drawings.

\begin{restatable}{corollary}{HH}
	For $n=9,10$ and $n \geq 13$, the number of crossings in any balanced restricted 3-circle drawing of $K_n$ exceeds $H(n)$, i.e., $\bcrN{3}(K_n)>H(n)$.
\end{restatable}
For $n\leq 7$, balanced restricted $3$-circle drawings of $K_n$ achieve the Harary-Hill bound, and we give drawings for $K_6$ and $K_7$. For $8 \le n \le 11$, we conclude that there exist unbalanced restricted 3-circle drawings of $K_n$ with $H(n)$ crossings.

\subsection*{Organization}
The remainder of our paper is organized as follows: 
In \cref{sec:notation}, we introduce tools to count the number of the crossings, which we then use in \cref{sec:3} to prove \cref{th:general} and \cref{th:balanced}.
We discuss the connection to the Harary-Hill conjecture in \cref{sec:HH} and conclude with a list of open problems in \cref{sec:open}.

\section{Tools for counting the number of  crossings}\label{sec:notation}

A \emph{simple drawing} of $G$ is a drawing where no edge crosses itself, two edges that share a vertex do not cross, and two edges with no shared vertices intersect at most once. Drawings that minimize the number of crossings are simple, so we only consider simple drawings.
In a tripartite-circle drawing of $K_{m,n,p}$, we label the three circles $\mysc{m}$, $\mysc{n}$, and ~$\mysc{p}$, and their numbers of vertices are $m$, $n$, and $p$, respectively. Consider \cref{fig:MNP}.  
Note that this drawing can be transformed by a projective transformation of the plane such that any one circle encloses the other two.
Therefore, without loss of generality  we consider drawings where the \emph{outer} circle $\mysc{p}$ contains the \emph{inner} circles $\mysc{m}$ and $\mysc{n}$. 
In such a drawing, we label the vertices on circles $\mysc{m}$ and $\mysc{n}$ in clockwise order and the vertices on circle $\mysc{p}$ in counterclockwise order.  Likewise, we read arcs of circles in clockwise order for inner circles and in counterclockwise order for outer circles.
\begin{figure}[htb]
	\centering
	\includegraphics[page=1]{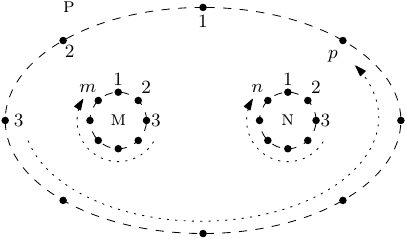}
	\caption{The vertices on the circles $\mysc{m}$ and $\mysc{n}$ are labeled clockwise; the vertices of the circle $\mysc{p}$ are labeled counterclockwise.}
	\label{fig:MNP}
\end{figure}

\subsection{Defining the $x$-labels}\label{xi}
For simplicity and without loss of generality, both papers \cite{AFS} and \cite{RT} considered simple bipartite-circle drawings of the complete bipartite graph where the two circles are assumed to be nested. Their results rely on the assignment of a vertex $\vtx{x}{i}{a}{b}$ on the outer circle $\mysc{b}$ for each vertex $i$ on the inner circle $\mysc{a}$. Because we are dealing with three circles and a pair of them is not necessarily nested, we adapt this definition as follows. 

Let $i$ be a vertex on circle $\mysc{a}$. The star formed by all edges from $i$ to $\mysc{b}$ together with circle $\mysc{b}$ partitions the plane into several disjoint regions, as shown in \cref{fig:defxi}. Exactly one of these  regions contains circle $\mysc{a}$. Such a region is enclosed by two edges from $i$ to $\mysc{b}$ and an arc on $\mysc{b}$ between two consecutive vertices. We define the second of these vertices (in clockwise or counterclockwise order depending on whether  $\mysc{b}$ is an inner or outer circle, respectively) as $\vtx{x}{i}{a}{b}$. If the two circles are clear from the context, we may also write $x_i$.

\begin{figure}[htb]
	\centering\begin{subfigure}[c]{.3\textwidth}
		\centering
		\includegraphics[page=3]{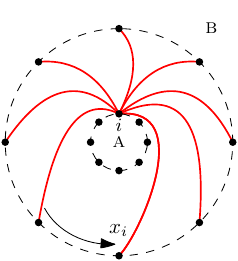}
		\subcaption{Nested: $\mysc a$ inside $\mysc b$}
	\end{subfigure}
	\hfil
	\begin{subfigure}[c]{.3\textwidth}
		\centering
		\includegraphics[page=4]{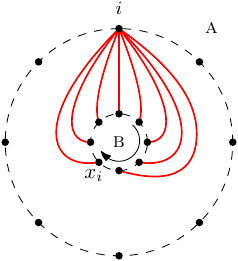}
		\subcaption{Nested: $\mysc b$ inside $\mysc a$}
	\end{subfigure}
	\hfil
	\begin{subfigure}[c]{.3\textwidth}
		\centering
		\includegraphics[page=5]{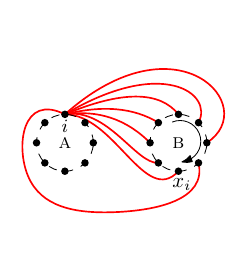}
		\subcaption{Non-Nested: $\mysc{a}$ beside $\mysc b$}
	\end{subfigure}
	\caption{Definition of vertex $\vtx{x}{i}{a}{b}$.}
	\label{fig:defxi}
\end{figure}

Abrego et al.\ \cite{AFS} observed that the $x$-labels are weakly ordered and suffice to describe the drawing up to isomorphism. Because we number the vertices on the outer circle in counterclockwise order (opposite to how it is done in  \cite{AFS}), the ordering of $x$-labels on the circle is reversed when compared to Lemma 1.4 from~\cite{AFS}. In particular, our weak ordering stated below is achieved, following the proof from~\cite{AFS}, by possibly renumbering {the} inner vertices.

\begin{lemma}[Lemma 1.4, \cite{AFS}] \label{lem:xOrder}
	Consider a simple bipartite-circle drawing of $K_{a,b}$ where the circles $\mysc{a}$ and $\mysc{b}$ have $a$ and $b$ vertices, respectively and the vertices are labeled so that {$x_1=b$}. Then it holds that 
	\[x_1\geq x_2\geq \cdots\geq x_a.\]
	Moreover, for a given sequence $(s_i)_i$ with $s_1\geq s_2\geq \cdots\geq s_a$, up to isomorphism, there is a unique simple bipartite-circle drawing of $K_{a,b}$ with
	$x_i = s_i$ for all $i \in \mysc{a}$. 
\end{lemma}

\subsection{Defining the $y$-labels}
As observed in \cref{lem:xOrder}, the $x$-labels are sufficient to describe a bipartite-circle drawing. We now aim to describe tripartite-circle drawings. We therefore introduce a new vertex assignment, $\vtx{y}{i}{a}{b}$ that depends on all three circles. See \cref{fig:defyiLP}.
Let $\mysc{a}$, $\mysc{b}$, and $\mysc{c}$ be the three circles and $i$ be a vertex on $\mysc{a}$. The star formed by all edges from $i$ to $\mysc{b}$ together with circle $\mysc{b}$ partitions the plane into disjoint regions. Exactly one of these regions contains the third circle $\mysc{c}$. This region is enclosed by two edges incident to $i$ and the arc between two consecutive vertices on $\mysc{b}$. We define the second of these two vertices (in clockwise or counterclockwise order depending on whether  $\mysc{b}$ is an inner or outer circle, respectively) as $\vtx{y}{i}{a}{b}$.

\begin{figure}[htb]
	\centering
	\begin{subfigure}[c]{.45\textwidth}
		\centering
		\includegraphics[page=6]{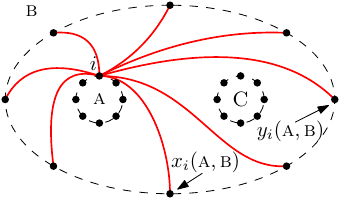}
		\subcaption{Nested: $\mysc a$ is inside $\mysc b$}
	\end{subfigure}\hfil
	\begin{subfigure}[c]{.45\textwidth}
		\centering
		\includegraphics[page=7]{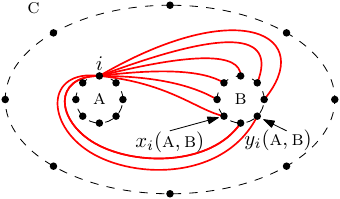}
		\subcaption{Non-Nested: $\mysc a$ is beside $\mysc b$}
	\end{subfigure}
	\caption{Illustration of $\vtx{y}{i}{a}{b}$  for the cases when (a) $\mysc a$ is inside $\mysc b$ and (b) when $\mysc a$ is besides~$\mysc b$.}
	\label{fig:defyiLP}
\end{figure}

\subsection{Counting crossings using $x$- and $y$-labels} 
If two edges $ab$ and $cd$ cross, then at least two (nonadjacent) vertices in $\{a,b,c,d\}$ are on the same circle. Hence, there are six total types of crossings between edges $ab$ and $cd$:
\begin{center}
	\begin{tabular}{lll}
		\mysc{mp/mp}-crossings: &$a$ and $c$ lie on \mysc{m},& and $b$ and $d$ lie on \mysc{p};\\
		\mysc{np/np}-crossings: &$a$ and $c$ lie on \mysc{n},& and $b$ and $d$ lie on \mysc{p};\\
		\mysc{mn/mn}-crossings: &$a$ and $c$ lie on \mysc{m},& and $b$ and $d$ lie on \mysc{n};\\
		\mysc{mn/mp}-crossings: &$a$ and $c$ lie on \mysc{m},& $b$ lies on \mysc{n},\hspace{.1in} and $d$ lies on \mysc{p};\\
		\mysc{mn/np}-crossings: &$a$ and $c$ lie on \mysc{n},& $b$ lies on \mysc{m},\hspace{.1in} and $d$ lies on \mysc{p};\\
		\mysc{mp/np}-crossings: &$a$ and $c$ lie on \mysc{p},& $b$ lies on \mysc{m},\hspace{.1in} and $d$ lies on \mysc{n}.
	\end{tabular}
\end{center}

We typically color the edges between each pair of circles with the same color, using three different colors for the different pairs. The first three types of crossings above only involve two circles and these are called \emph{monochromatic} crossings. The last three types involve all three circles with edges of different colors. Thus, these crossings are called \emph{bichromatic} crossings. We use the $x$- and $y$-labels to count the monochromatic and bichromatic crossings, respectively. The following definitions are used throughout the rest of the paper.

For vertices $k$ and $\ell$ on a circle with $n$ vertices numbered $1,\ldots, n$ clockwise (respectively, counterclockwise), let
\begin{equation*}
	d_n(k,\ell):=\ell-k\mod n
\end{equation*}
denote the distance from $k$ to $\ell$ in clockwise (respectively, counterclockwise) order on the circle. Let $[n]:=\lbrace 1,2,\ldots,n\rbrace$. For any $u,v \in [n]$, define
$$f_n(u,v) := \binom{d_n(u,v)}{2} + \binom{n-d_n(u,v)}{2}.$$

For vertices $i$ and $j$ on the inner (respectively, outer) circle $\mysc{a}$, we use $[i,j]$ to denote the arc of $\mysc{a}$ read clockwise (respectively, counterclockwise) from $i$ to $j$. We include $i$ and $j$ in the interval $[i,j]$, whereas $(i,j)$ does not include $i$ and $j$. We similarly define $[i,j)$ and $(i,j]$.

\subsubsection{Counting crossings involving two circles}

We start by stating the following result from \cite{RT} to take care of the monochromatic crossings. 

\begin{lemma}[\cite{RT}{, Sect. 2}]\label{eq:MonoCrossings}
	The number of crossings in a simple bipartite-circle drawing of the complete bipartite graph~$K_{m,n}$ is
	\begin{equation*}
		\sum_{1\leq i<j\leq m}f_n(x_i,x_j).
	\end{equation*}
\end{lemma}

\subsubsection{Counting crossings involving three circles}

The following lemma introduces a means of counting all three types of bichromatic crossings using the $y$-labels. See \cref{fig:yifigure} for a visual representation of a possible $\mysc{m}\mysc{p}/\mysc{n}\mysc{p}$-crossing.

\begin{figure}[htb]
	\centering
	\includegraphics[page=2]{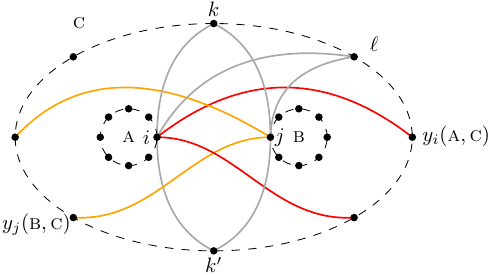}
	\caption{Illustration for the case that $\mysc{(a,b,c)}=\mysc{(m,n,p)}$. Since $k'\in I_1$ and $\ell\in I_2$, edges between the vertices $i$, $j$, $k'$, and $\ell$ do not cross, but edges between the vertices $i$, $j$, $k$ and $\ell$ do cross since $k$ and $\ell$ are in the same interval~$I_2$.}
	\label{fig:yifigure}
\end{figure}

\begin{lemma}\label{yilemma1}\label{eq:BiCrossings}
	Let $\mysc{a}$, $\mysc{b}$, and $\mysc{c}$ be three disjoint circles with the disjoint vertex sets $\{1,\dots, a\}$, $\{1,\dots, b\}$, and $\{1,\dots, c\}$, respectively. Then the number of $\mysc{a}\mysc{c}/\mysc{b}\mysc{c}$-crossings is given by
	\begin{equation*}
		\sum_{\substack{1\leq i\leq a\\ 1\leq j\leq b}} f_c\big(\vtx{y}{i}{a}{c},\vtx{y}{j}{b}{c}\big).
	\end{equation*}
\end{lemma}

\begin{proof}
	Fix a vertex $i$ on $\mysc{a}$ and a vertex $j$ on $\mysc{b}$ and consider the corresponding vertices $\vtx{y}{i}{a}{c}$ and $\vtx{y}{j}{b}{c}$ on circle $\mysc{c}$, see \cref{fig:yifigure}.
	For every pair of distinct vertices $k$ and $\ell$ both in the interval $[\vtx{y}{i}{a}{c}, \vtx{y}{j}{b}{c})=:I_1$ on $\mysc{c}$ there is exactly one crossing among edges $ik$ and $i\ell$, and $jk$ and $j\ell$. Similarly, there is exactly one crossing among the edges $ik$ and $i\ell$, and edges $jk$ and $j\ell$ when $k$ and $\ell$ are in $[\vtx{y}{j}{b}{c}, \vtx{y}{i}{a}{c})=:I_2$. Moreover note that if a vertex $k$ is in $I_1$ and a vertex $\ell$ is in $I_2$ then there are no crossings among edges $ik$, $i\ell$, $jk$ and $j\ell$. Consequently,  there are exactly $f_c\big(\vtx{y}{i}{a}{c},\vtx{y}{j}{b}{c}\big)$ crossings among edges incident with vertices $i$ and $j$. Therefore the total number of $\mysc{a}\mysc{c}/\mysc{b}\mysc{c}$-crossings is as claimed.
\end{proof}

\subsubsection{Total crossing count}
The number of crossings in a simple tripartite-circle drawing of  
$K_{m,n,p}$ can be found by counting the crossings in the three different, simple bipartite-circle drawings of $K_{m,n}$, $K_{m,p}$ and $K_{n,p}$, along with crossings involving all three circles. 
Therefore, we say
a \emph{cyclic assignment} of \mysc{(a,b,c)} to \mysc{(m,n,p)} is one triple in the  set $\mathfrak t:= \{\mysc{(m,n,p)}, \mysc{(n,p,m)}, \mysc{(p,m,n)}\}$, with 
the number of vertices on the circles \mysc{a, b}, and \mysc{c} denoted by $a,~b$, and $c$, respectively.  
\begin{theorem}\label{th:yitheorem}
	The number of crossings in a simple tripartite-circle drawing of~ $K_{m,n,p}$ is given by 
	\[
	\sum_{\mysc{(a,b,c)}\in \mathfrak t} \left(
	\sum_{\substack{i < j\\i,j \in \mysc{A}}} f_b\big(\vtx{x}{i}{a}{b},\vtx{x}{j}{a}{b}\big)
	+\sum_{\substack{i \in \mysc{A}\\j \in \mysc{B}}} f_c\big(\vtx{y}{i}{a}{c},\vtx{y}{j}{b}{c}\big)
	\right).
	\]
\end{theorem}

\begin{proof}
	The monochromatic crossings are counted by the first expressions in the brackets using \cref{eq:MonoCrossings}. The second expression corresponds to the bichromatic crossings using \cref{yilemma1}.
\end{proof}

\section{Bounding the tripartite-circle crossing number---\\Proofs of Theorem \ref{th:general} and Corollary \ref{th:balanced}}\label{sec:3}
In this section, we prove the upper and lower bounds of \cref{th:general,th:balanced}. We start with the lower bounds and then proceed with the upper bounds.
\subsection{Lower bounds}
To prove the lower bounds, we start with two lemmas.

\begin{lemma}\label{le:min}
	The function $f_n(a,b)$ attains its minimum $M$ if and only if $|a-b| \in \{\floor*{n/2}, \ceil*{n/2}\}$. Among pairs $(a,b)$ such that $|a-b| \notin \{\floor*{n/2}, \ceil*{n/2}\}$, the minimum of $f_n$ exceeds $M$ by $1$ if $n$ is even and by $2$ if $n$ is odd.
\end{lemma}

\begin{proof}
	First note that $d_n(a,b) = \floor*{n/2}$ if and only if $|a-b| \in \{\floor*{n/2}, \ceil*{n/2}\}$ because $n = \floor*{n/2} + \ceil*{n/2}$.
	
	Consider the auxiliary {real-valued} function $g_n(x) = \binom{x}{2} + \binom{n-x}{2} = (x-\frac{n}{2})^2+\frac{n^2-2n}{4}$. It is a quadratic function minimized at $\frac{n}{2}$, symmetric about $x=\frac{n}{2}$, decreasing for $x < \frac{n}{2}$, and increasing for $x > \frac{n}{2}$. 
	Consequently, the minimum value of $g_n$ 
	over the integers is attained at $x \in \{\floor*{n/2}, \ceil*{n/2}\}$; the next smallest value of $g_n$ over the integers is attained at $x \in \{\floor*{n/2}-1, \ceil*{n/2}+1\}$.
	Finally, the claim follows by a computation showing that
	\begin{align*}g_n(\floor*{n/2}-1) - g_n(\floor*{n/2})
		&= \begin{cases}2 & \text{if $n$ is odd,}\\
			1 & \text{if $n$ is even.}\end{cases}\end{align*}
	This concludes the proof.
\end{proof}

For every cyclic assignment of $\mysc{(a,b,c})$ to $\mysc{(m,n,p})$, we denote the minimum number of \mysc{a}\mysc{b}/\mysc{a}\mysc{c}-crossings among all simple tripartite-circle drawings of~$K_{m,n,p}$ by $\crN{3}^\mysc{a}(K_{m,n,p})$.

\begin{lemma}\label{le:mnpp}
	For every cyclic assignment of $\mysc{(a,b,c})$ to $\mysc{(m,n,p})$, 
	$$\crN{3}^{\mysc{c}}(K_{m,n,p}) \geq ab \left\lfloor \frac{c}{2}\right\rfloor\bigg\lfloor \frac{c-1}{2}\bigg\rfloor .$$
\end{lemma}
\begin{proof}
	From the count  in Lemma \ref{yilemma1}, and its minimization in  Lemma \ref{le:min}, we have the lower bound
	\begin{align*}
		\sum_{\substack{1\leq i\leq a\\ 1\leq j\leq b}} f_c\big(\vtx{y}{i}{a}{c},\vtx{y}{j}{b}{c}\big) 
		&\geq\sum_{\substack{1\leq i\leq a\\ 1\leq j\leq b}}
		\binom{\lfloor c/2\rfloor}{2}+
		\binom{\lceil c/2\rceil}{2}
		=ab\left[
		\binom{\lfloor c/2\rfloor}{2}+
		\binom{\lceil c/2\rceil}{2}
		\right]
	\end{align*}
	Adding the last two terms directly yields the claim. 
\end{proof}

This allows us to find a lower bound on $\crN{3}(K_{m,n,p})$. 
\begin{proof}[Proof of the lower bound of \autoref{th:general}]
	In \cref{th:yitheorem} we count the total number of crossings in a simple tripartite-circle drawing of $K_{m,n,p}$. 
	\begin{equation*}
		\crN{3}(K_{m,n,p})\geq \crN{2}(K_{m,n})+\crN{2}(K_{m,p})+\crN{2}(K_{n,p})
		+\crN{3}^{\mysc{m}}(K_{m,n,p})+\crN{3}^{\mysc{n}}(K_{m,n,p})+\crN{3}^{\mysc{p}}(K_{m,n,p})
	\end{equation*}
	\cref{le:mnpp} gives a lower bound on the last three summands and directly implies the claimed lower bound:
	\[ \crN{3}(K_{m,n,p})\geq 
	\sum_{\mysc{(a,b,c)}\in \mathfrak t}
	\left( \crN{2}(K_{a,b})
	+ \ ab\bigg\lfloor \frac{c}{2}\bigg\rfloor
	\bigg\lfloor \frac{c-1}{2}\bigg\rfloor
	\right). \qedhere\]
\end{proof}

\subsubsection{Improving the lower bound}
With the help of the following lemmas, the lower bound can be improved by 2. {This improvement is relevant in the analysis of small graphs and the connection to the Harary-Hill conjecture. It is also used in \cite{K22n} to settle the tripartite-circle crossing number of $K_{2,2,n}$.}

\begin{lemma}[Special Inversion Lemma]\label{lem:SPECIALinversion}
	Fix the placement of two circles $\mysc{a}$ and $\mysc{b}$ inside circle $\mysc{c}$. For a tripartite-circle drawing $D$, $\vtx xiac = \vtx yiac$ for all $i$ on $\mysc a$ if and only if $\vtx yjca = \vtx ykca$ for all $j,k$ on $\mysc c$.
\end{lemma}
\begin{proof}
	Let $D^*$ be the restriction of $D$ to edges between $\mysc a$ and $\mysc c$.
	We refer to the connected components of the complement of $D^*$ as {\it faces}.
	We start with an observation.
	By the definitions of $\vtx xiac$ and $\vtx yiac$, we have $\vtx xiac = \vtx yiac$ if and only if the pair of incident edges $\{i,\vtx xiac\}$ and $\{i,\vtx xiac+1\}$ divides the interior of $\mysc{c}$ into two parts where $\mysc{a}$ and $\mysc{b}$ lie in the same part and all other edges from $i$ lie in the other part. This holds for all $i$ if and only if no edge of $D^*$ separates $\mysc{a}$ from $\mysc{b}$; in other words, the circle $\mysc{b}$ lies in a face $F$ of $D^*$ adjacent to $\mysc{a}$.
	\begin{figure}[htb]
		\centering
		\includegraphics[page=1]{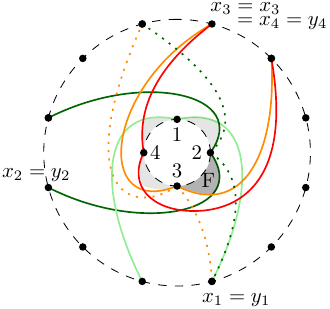}
		\caption{Depiction of special $y$-inversion. }
		\label{fig:Yinversion}
	\end{figure}
	\cref{fig:Yinversion} shows an example with $\vtx xiac = \vtx yiac$ for all $i$ and with the faces adjacent to $\mysc{a}$ shaded in gray. Circle $\mysc{b}$ lies in one of these faces, say $F$. This finishes our observation.
	
	Now we prove the lemma. Suppose $\vtx xiac=\vtx yiac$ for all $i$ on $\mysc{a}$. By the above observation, the circle $\mysc{b}$ lies in a face $F$ of $D^*$ adjacent to $\mysc{a}$. Let $q$ and $q+1$ be the two vertices on $\mysc{a}$ and on the boundary of $F$, where $q+1$ comes clockwise directly after $q$. For each $j$ on $\mysc{c}$, the triangle $T_j = (j,q,q+1)$ cannot cross $F$, since it is a face, and shares the side $\{q,q+1\}$ with $F$. By the properties of simple drawings, either $F$ is in the interior of $T_j$ and all other edges from $j$ are on the exterior, or $F$ is on the exterior of $T_j$ and all other edges from $j$ are in the interior of $T_j$. In either case, $\vtx yjca = q+1$ for all $j$ on $\mysc{c}$.
	
	For the converse, suppose $\vtx yjca = \vtx ykca = q+1$ for all $j,k$ on $\mysc{c}$ and some $q$ on $\mysc{a}$. Then for each $j$ on $\mysc{c}$, the triangle $T_j = (j,q,q+1)$ is the smallest triangle (by containment) from $j$ to $\mysc{a}$ that encloses $\mysc{b}$. By the minimality of $T_j$, no edge from $j$ in $D^*$ can cross the interior of $T_j$. The intersection of these triangles over all $j$ on $\mysc{c}$ is a face $F$ of $D^*$ containing $\mysc{b}$ and adjacent to $\mysc{a}$ at $\{q,q+1\}$. By the observation, $\vtx xiac = \vtx yiac$ for all $i$ on $\mysc{a}$.
\end{proof}

For vertices $x$ and $y$ on a circle with $n$ vertices, we define $$\mind_n(x,y):=\min\{d_n(x,y), d_n(y,x)\}.$$
The following lemma analyzes the situation when a term counting bichromatic crossings is minimized.
\begin{lemma}\label{le:ycluster}
	Consider a tripartite-circle drawing
	with circles \mysc{a}, \mysc{b}, and \mysc{c} with $c$ vertices on \mysc{c}. If 
	{$f_c(\vtx yiac,\vtx yjbc)=\min_{(u,v)\in [c]^2} f_c(u,v)$} for every $i \in \mysc{a}$ and $j \in \mysc{b}$, then there are vertices $u_\mysc{a}, u_\mysc{b}\in \mysc c$ with the following properties:
	\begin{enumerate}
		\item For both circles $\mysc d \in \{\mysc{a,b}\}$, $\vtx yidc \in \{u_\mysc{d}, u_\mysc{d}+1\}$ for all $i \in \mysc d$. 
		\item For some circle $\mysc d \in \{\mysc{a,b}\}$, $\vtx yidc = u_\mysc{d}$ for all $i \in \mysc{d}$. If $c$ is even, then for both circles $\mysc d \in \{\mysc{a,b}\}$, $\vtx yidc = u_\mysc{d}$ for all $i \in \mysc{d}$.
	\end{enumerate}
\end{lemma}
\begin{proof}
	Suppose that $f_c(\vtx yiac,\vtx yjbc)$ {equals the minimum value of $f_c$} for every $i \in \mysc{a}$ and $j \in \mysc{b}$. By \cref{le:min},  it holds that
	$\big|\vtx yiac - \vtx yjbc\big| \in \big\{\floor{c/2},\ceil{c/2}\big\}$ and 
	$ \mind_c\big(\vtx yiac, \vtx yjbc\big) = \floor{c/2}$.
	
	First we prove property (i). Without loss of generality, we assume that $\mysc d = \mysc{a}$. Let $v := y_j(\mysc{b},\mysc{c})$ for some $j \in \mysc{b}$. The vertex $v$ on $\mysc{c}$ can have only one vertex at distance $\floor{c/2} = c/2$ if $c$ is even and only two vertices at distance $\floor{c/2}$, adjacent to one another, if $c$ is odd. Since $\mind_c(\vtx yiac, v) = \floor{c/2}$ for all $i \in \mysc{a}$, we get that the $\vtx yiac$ labels are all the same in the even case, or all on one of two adjacent vertices in the odd case.
	
	Next we prove property (ii). Suppose for contradiction that $\vtx yiac = u_\mysc{a}$, $\vtx yjac = u_\mysc{a} + 1$, $\vtx ykbc = u_\mysc{b}$, and $\vtx y\ell bc = u_\mysc{b} + 1$ for some $i,j \in \mysc{a}$ and $k,\ell \in \mysc{b}$.
	Since $\mind_c(\vtx yiac, \vtx yjbc) = \floor{c/2}$ for every $i \in \mysc{a}$ and $j \in \mysc{b}$, we get that $\mind_c(u_\mysc{a},u_\mysc{b}) = \mind_c(u_\mysc{a},u_\mysc{b}+1) = \floor{c/2}$. No vertex on $\mysc{c}$ other than $u_\mysc{a}$ is at distance $\floor{c/2}$ from both $u_\mysc{b}$ and $u_\mysc{b}+1$, but $u_\mysc{a}+1 \ne u_\mysc{a}$ must also be at distance $\floor{c/2}$ from both.
\end{proof}

Now, we use these insights in order to improve the lower bound of \autoref{th:general}.
\begin{corollary}[Improvement of lower bound of \autoref{th:general}] \label{cor:improvedLower}
	{Let $\mathfrak t:=\{(m,n,p),(n,p,m),(p,m,n)\}$. Then,  for any integers $m, n, p \geq 3$, it holds that}
	\begin{align*}
		\crN{3}(K_{m,n,p})\geq
		\sum_{(a,b,c)\in \mathfrak t}
		\left(\crN{2}(K_{a,b})
		+ \ ab\left\lfloor \frac{c}{2}\bigg\rfloor
		\bigg\lfloor \frac{c-1}{2}\right\rfloor
		\right) + 2.
	\end{align*}
\end{corollary}

\begin{proof}
	Recall that the lower bound  in Theorem \ref{th:general}  was obtained by simultaneously minimizing all six terms in the formula of \cref{th:yitheorem}. 
	Suppose a bichromatic crossing count $\sum_{\substack{i \in \mysc{a}\\j \in \mysc{b}}} f_c(\vtx yiac,\vtx yjbc)$ attains its minimum. Lemma \ref{le:ycluster}~(ii) implies without loss of generality that all $\vtx yiac$ labels are equal. \cref{lem:SPECIALinversion} then implies that $\vtx xjca = \vtx yjca$ for all $j$ on $\mysc{c}$ (and $\vtx xjcb = \vtx yjcb$ for all $j \in \mysc{c}$ if $c$ is even). To achieve the minimum number of monochromatic crossings between $\mysc{a}$ and $\mysc{c}$, the $\vtx xjca$ labels must be equally spaced around $\mysc{a}$ as already observed in \cite{RT}. Then, since $\vtx yjca=\vtx xjca$, the $\vtx yjca$ labels are also equally spaced. Since $a \ge 3$, the $\vtx yjca$ labels are on more than two points. By Lemma \ref{le:ycluster}, the term
	$\sum_{\substack{i \in \mysc{b}\\j \in \mysc{c}}} f_a(\vtx yiba,\vtx yjca)$
	does not attain its minimum.
	
	For $c$ even, Lemma \ref{le:ycluster}~(ii) further implies that all the $\vtx yibc$ labels are equal. By \cref{lem:SPECIALinversion}, $\vtx xjcb = \vtx yjcb$ for all $j \in \mysc{c}$. If the minimum number of monochromatic crossings between $\mysc{b}$ and $\mysc{c}$ is achieved, then the bichromatic crossings term
	$\sum_{\substack{i \in \mysc{a}\\j \in \mysc{c}}} f_b(\vtx yiab,\vtx yjcb)$
	also does not attain its minimum.
	By \cref{le:min}, if $c$ is odd then at least one of the six terms is at least 2 more than its minimum. If $c$ is even then at least two of the six terms are at least 1 more than their minima. Regardless of the parity of $c$, the lower bound given by minimizing all six terms simultaneously can be improved by 2.
\end{proof}

\subsection{Upper bounds}
In this subsection, we provide drawings that settle the upper bounds of \cref{th:general,th:balanced}.
We define the drawing within a small stripe around the equator of the sphere and visualize it by a rectangle where 
the left and right boundaries are identified. 
Consider \cref{fig:genConstrRealizGen} for an illustration. 
In contrast to before, in the following drawings {the interiors of the three circles are disjoint.}
%
\begin{figure}[htb]
	\centering
	\includegraphics[page=5]{proofIDEA3n}
	\caption{Illustration of the construction of a tripartite-circle drawing of~$K_{m,n,p}$ with $m=4,n=5,p=6$ on a small strip around the equator of a sphere, the left and right boundaries of the rectangle are identified. The subdrawing induced by the circles $\mysc{a}:=\mysc{m}$ and $\mysc{b}:=\mysc{n}$ is highlighted by colors. }\label{fig:genConstrRealizGen}
\end{figure}

\subparagraph{Definition (Linear Description)} 
We start by defining the subdrawing $D'$ induced by the vertices of two distinct circles~\mysc{a} and \mysc{b}. Let \mysc{a} be a circle with $a$ vertices and \mysc{b} a disjoint circle with $b$ vertices. We assume that \mysc{a}  is strictly left of \mysc{b}. The vertices are placed on the circles such that $\left\lceil\frac{a}{2}\right\rceil$ and $\left\lceil\frac{b}{2}\right\rceil$ vertices lie in the closed top halves of \mysc{a} and \mysc{b}, respectively; the vertices on \mysc{a} are labeled clockwise by $\{1,\dots,a\}$ starting with the clockwise first vertex in the closed top half, while the vertices on \mysc{b} are labeled counterclockwise by $\{1,\dots,b\}$ starting with the counterclockwise first vertex in the closed top half. 
Let $\ell_1$ and $\ell_2$ be two vertical lines separating \mysc{a} and \mysc{b} where $\ell_1$ is strictly left of $\ell_2$.
\begin{enumerate}
	\item On $\ell_1$ we mark $a\cdot b$ points, which are labelled by $a_{i,j}$  for $i\in[a]$ and $j\in[b]$ such that the indices increase lexicographically from top to bottom. Each $a_{i,j}$ belongs to an edge of vertex $i$ on $\mysc{a}$; between vertex $i$ and $a_{i,j}$ the edge is realized by some $x$- and $y$-monotone curve $e^1_{i,j}$. Moreover, no two curves $e^1_{i,j}$ intersect.
	\item On $\ell_2$ we mark $a\cdot b$ points, which are labelled by $b_{i,j}$  for $i \in [b]$ and $j\in[a]$ such that the indices increase lexicographically from top to bottom. Each $b_{i,j}$ belongs to an edge of vertex $i$ on $\mysc{b}$;  
	the edge between these two points is realized by some $x$- and $y$-monotone curve $e^2_{i,j}$. Moreover, no two curves $e^2_{i,j}$ intersect.
	\item Between $\ell_1$ and $\ell_2$, we connect $a_{i,j}$ and $b_{j,i}$ by a straight-line segment.
\end{enumerate}
The drawing $D$ is obtained by constructing a drawing $D'$ for each pair of circles and overlaying them.
By construction, the drawing $D$ has the following properties:
\begin{enumerate}
	\item each edge is $x$-monotone,
	\item the drawing is partitioned into six vertical stripes; within each stripe every edge is $x$- and $y$-monotone, 
	\item there exist two types of stripes, either containing $\mysc{ab/ab}$-crossings or $\mysc{ab/bc}$-crossings, and
	\item each edge is contained in three stripes.
\end{enumerate}
These properties imply the following fact. 

\begin{proposition}
	The drawing $D$ is simple.
\end{proposition}
\begin{proof}
	For each pair of edges, there exists a unique stripe where the two edges potentially cross. Since, by property (ii), the edges are $x$- and $y$-monotone within each stripe, any pair of edges have at most one point in common. Because edges are $x$-monotone, no edge crosses itself. Thus, $D$ is a simple drawing.
\end{proof}

It remains to analyze the number of crossings {for which we present two proofs. Proof 1 uses the fact that the drawings are simple together with \cref{eq:MonoCrossings,eq:BiCrossings} and offers an insight on what is needed to improve the construction.
	The crossing number can also be directly computed, as shown in Proof 2, providing an upper bound even if the drawings were not simple.}

\begin{proposition}
	For any integers $m, n, p \geq 3$, let $\mathfrak t:=\{(m,n,p),(n,p,m),(p,m,n)\}$.
	The number of crossings in the drawing~$D$ is
	$$\sum_{\substack{(a,b,c)\in\mathfrak t}}
	\left(
	\binom{a}{2}\binom{b}{2}
	+ \ ac\left\lfloor \frac{b}{2}\bigg\rfloor
	\bigg\lfloor \frac{b-1}{2}\right\rfloor
	\right).$$
\end{proposition}
\begin{proof}[Proof 1]
	It is easy to see from the construction that
	$\vtx{x}{i}{a}{b}=\vtx{y}{i}{a}{b}$, $\vtx{x}{i}{b}{a}=\vtx{y}{i}{b}{a}$, and $d_b(\vtx{y}{i}{a}{b},\vtx{y}{j}{c}{b})=\lfloor\frac{b}{2}\rfloor$. 
	Consequently, by \cref{eq:MonoCrossings}, the number of crossings of type \mysc{ab/ab} is
	\begin{equation*}
		\sum_{1\leq i<j\leq a}f_b(\vtx{x}{i}{a}{b},\vtx{x}{j}{a}{b})=\binom{a}{2}\binom{b}{2}.
	\end{equation*}
	By \cref{eq:BiCrossings}, the number of type \mysc{ab/bc} is
	\begin{equation*}
		\sum_{\substack{1\leq i\leq a\\ 1\leq j\leq c}} f_b\big(\vtx{y}{i}{a}{b},\vtx{y}{j}{c}{b}\big)
		=ac\left(\binom{\left\lfloor\frac{b}{2}\right\rfloor}{2}+\binom{\left\lceil\frac{b}{2}\right\rceil}{2}\right)
		=ac
		\bigg\lfloor\frac{b}{2}\bigg\rfloor
		\bigg\lfloor\frac{b-1}{2}\bigg\rfloor.
	\end{equation*}
	This finishes the first proof.
\end{proof}

\begin{proof}[Proof 2]
	Alternatively, we count the number of crossings directly.
	By definition, the \mysc{ab/ab} crossings occur between $\ell_1$ and $\ell_2$; in this part of the drawing the edges are straight-line segments. Any pair of vertices on circle \mysc{a} and any pair of vertices on circle \mysc{b} together form exactly one crossing. We have $\binom{a}{2}\binom{b}{2}$ crossings.
	
	For the crossings of type  \mysc{ab/bc}, it suffices to count the bundle crossings. If two bundles cross they add $ac$ crossings. Moreover, it follows from the construction that two bundles cross if {and only if} they are both in the top or both in the bottom half. Consequently, the number of crossings is $$ac\left({\displaystyle\binom{\left\lfloor\frac{b}{2}\right\rfloor}{2}}+{\displaystyle\binom{\left\lceil\frac{b}{2}\right\rceil}{2}}\right).$$
	As shown above, this evaluates to $ac
	\left\lfloor\frac{b}{2}\right\rfloor
	\left\lfloor\frac{b-1}{2}\right\rfloor$ and therefore finishes the second proof of the proposition. 
\end{proof}

As the proposition was the last missing item, this finishes the proof of the upper bound and thus of \cref{th:general}. Note that this construction achieves the minimum possible number of bichromatic crossings by Lemma \ref{le:mnpp}.

\subsection{Balanced case}
\cref{th:general} and \cref{cor:improvedLower} imply \cref{th:balanced} for the special case of $m=n=p$.
\begin{proof}[Proof of \autoref{th:balanced}]
	For the lower bound, \autoref{th:general}  and \cref{cor:improvedLower} give that \[\crN{3}(K_{n,n,n}) \geq 3\crN{2}(K_{n,n}) + 3 n^2\left\lfloor\frac{n}{2}\right\rfloor\left\lfloor\frac{n-1}{2}\right\rfloor+2.\]
	With the bipartite cylindrical crossing number from  \autoref{eq:bipartitebalanced} we have
	\begin{align*}
		\crN{3}(K_{n,n,n}) &\geq 
		3n\binom{n}{3}+3n^2\bigg\lfloor\frac{n}{2}\bigg\rfloor\left\lfloor\frac{n-1}{2}\right\rfloor +2.
	\end{align*}
	
	For the upper bound, the construction includes drawings for $K_{n,n,n}$. In this case, we obtain highly symmetric drawings, which are especially appealing.
	In particular, such a drawing can be defined by two consecutive stripes; see \cref{fig:k444}. 
	The formula simplifies to 
	\[3\binom{n}{2}^2+3n^2\left\lfloor \frac{n}{2}\right\rfloor\left\lfloor \frac{n-1}{2}\right\rfloor.\qedhere\]
\end{proof} 

While the lower bound order is $\nicefrac{5}{4}\cdot n^4$, the upper bound order is $\nicefrac{6}{4} \cdot n^4$. Consequently, the bounds are fairly close. 
Moreover, instead of a linear representation, similar drawings can be defined in a \emph{cyclic} way, as shown in \cref{fig:genConstrRealizCYC}. 

\begin{figure}[htbp]
	\centering
	\includegraphics[page=1]{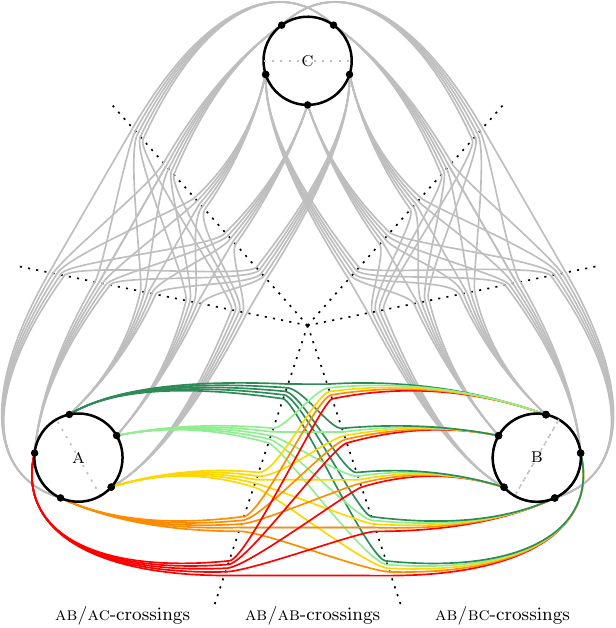}
	\caption{A tripartite-circle drawing of $K_{n,n,n}$ for $n=5$.}\label{fig:genConstrRealizCYC}
\end{figure}

\begin{remark}\label{rem:improvedUpperBound}
	By a slight modification, we improve the upper bound. To do so, we place at least one vertex on the intersection of the closed top and bottom half of the circle and route half of its incident edges via the upper half and the other half of its edges via the bottom half. This idea is used to construct the drawings in \cref{fig:better} (and \cref{fig:k444}).
	
	\begin{figure}[htbp]
		\centering
		\begin{subfigure}[c]{.485\textwidth}
			\centering
			\includegraphics[page=1]{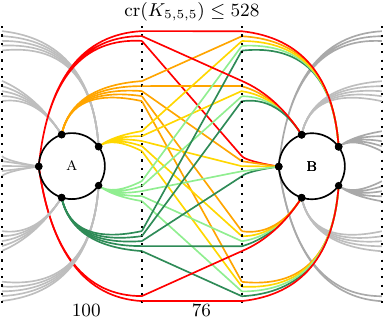}
			\subcaption{A drawing of $K_{5,5,5}$ with 528 crossings.}
		\end{subfigure}\hfil
		\begin{subfigure}[c]{.485\textwidth}
			\centering
			\includegraphics[page=3]{proofIDEA4}
			\subcaption{A drawing of $K_{6,6,6}$ with 1161 crossings.}
		\end{subfigure}
		\caption{}
		\label{fig:better}
	\end{figure}
	
	Depending on the parity of $n$, the number of monochromatic crossings between two circles is
	$$\begin{cases}
		\binom{n-1}2\binom{n}2+(n-1)\left(\binom{\floor{n/2}}{2} +\binom{\ceil{n/2}}{2}\right)  & n \text{ odd,}\\
		2\left(\binom{n/2-1}{2}n^2+\nicefrac{1}{2}(n-2)n^2+\nicefrac{1}{4}\cdot n^2\right)& n \text{ even,}\\
	\end{cases}$$
	while the number of bichromatic crossings is 
	$$\begin{cases}
		n^2\left(\binom{\floor{n/2}}{2} +\binom{\ceil{n/2}}{2}\right) & n \text{ odd,}\\
		\binom{(n-1)}{2}^2+2(n-1)\left(\binom{n/2-1}{2}+\binom{n/2}{2}\right) +(n/2)^2+(n/2-1)^2 & n \text{ even.}\\
	\end{cases}$$
	Consequently, multiplying by three and summing both terms, the number of crossings evaluates to
	$$\begin{cases}
		\nicefrac{3}{4} (2n^4-5n^3+3n^2+n-1) & n \text{ odd,}\\
		\nicefrac{3}{4}(2n^4 - 6 n^3 + 7 n^2) & n \text{ even.}\\
	\end{cases}$$
	Unfortunately, this improves only lower order terms, i.e., 
	the number of saved crossings is
	$$\begin{cases}
		\nicefrac{3}{4}(n^3-n^2-n+1) & n \text{ odd,}\\
		\nicefrac{3}{2}(n^3 - 3 n^2) & n \text{ even.}\\
	\end{cases}$$
\end{remark}

\subsubsection{Balanced case with few vertices}
In this section, we present numerical results and improved drawings of $K_{n,n,n}$ for small values of $n$. 
The values are summarized in \cref{tab:smallN}. 
We improve the upper bounds with concrete drawings and the lower bounds with \cref{cor:improvedLower} and the following fact:
\begin{proposition} \label{prop:lb}
	For any integers $m,n,$ and $p$,
	\[\crN{3}(K_{m,n,p})\geq cr(K_{m+n+p})-\binom{m}{4}-\binom{n}{4}-\binom{p}{4}.\]
\end{proposition}
\begin{table}[b!]
	\centering
	\caption{Bounds of $\ctcr(K_{n,n,n})$ for small $n$.}
	\label{tab:smallN}
	\begin{tabular}{|c|c c| c c|}
		
		\hline &Lower bound &Improved &Improved& Upper bound\\
		$n$ & \cref{th:balanced}& Lower bound& Upper bound &\cref{th:balanced}\\
		\hline
		\hline
		2 &   -- &  3 &     3 &   --\\
		3 &   38 & -- &    42 &    54\\
		4 &  146 & 147&   175 &   204\\
		5 &  452 & -- &   528 &   600\\
		6 & 1010 & -- &  1161 &  1323\\
		7 & 2060 & -- &  2430 &  2646\\
		8 & 3650 & -- &  4176 &  4656\\
		9 & 6158 & -- &  7296 &  7776\\
		10 & 9602& -- & 11025 & 12075\\
		\hline
	\end{tabular}
\end{table}

The proposition follows from the fact that a tripartite-circle drawing of {$K_{m,n,p}$}  yields a drawing of the complete graph $K_{m+n+p}$ by adding all straight-line segments within the three circles; see also \cref{sec:HH}.

In the following we explain how to obtain the bounds displayed in \cref{tab:smallN}.
For $n=2$, it holds that $\crN{3}(K_{2,2,2})\geq 3$ since $cr(K_6)=3$ and \cref{fig:K222} shows that three crossings can be attained.

\begin{figure}[htb]
	\begin{subfigure}[b]{.3\textwidth}
		\centering
		\includegraphics[page=3]{proofIDEA3n}
		\caption{An optimal drawing of $K_{2,2,2}$ with three crossings.}
		\label{fig:K222}
	\end{subfigure}
	\hfil
	\begin{subfigure}[b]{.3\textwidth}
		\centering
		\includegraphics[page=2]{proofIDEA3n}
		\caption{Two drawings of $K_{3,3,3}$ with  42~crossings.}
		\label{fig:333drawings}
	\end{subfigure}
	\hfil
	\begin{subfigure}[b]{.3\textwidth}
		\centering
		\includegraphics[page=1]{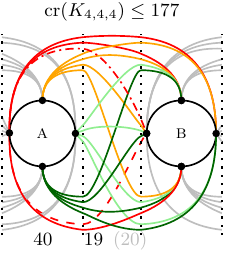}
		\caption{Two drawings of $K_{4,4,4}$, one with 177 and one with 180 crossings.}
		\label{fig:k444}
	\end{subfigure}
	\caption{
		{Drawings of $K_{n,n,n}$ for $n=2,3,4$ with few crossings.}
		In (b) and (c), the two drawings are obtained by considering either the dash dotted edge or the dotted edge.}
\end{figure}

For $n=3$, a lower bound of 38 follows from Corollary \ref{cor:improvedLower} and the upper bound of 42 from the drawing in \cref{fig:333drawings}.

In case $n=4$,  we use \cref{prop:lb} for the lower bound. Since $cr(K_{12})=150$, we obtain $\crN{3}(K_{4,4,4})\geq cr(K_{12})-3=147$. For the upper bound, \cref{fig:K444175} presents a drawing with 175 crossings. This drawing is obtained by a slight modification of the drawing corresponding to \cref{fig:k444} with the dash dotted edge. In particular, in the middle copy (orange), the long edge between the leftmost vertex of \mysc b and the rightmost vertex of \mysc c is drawn in the top half, while its corresponding edges in the other two copies are drawn in the bottom half. This saves the two crossings between the middle long edge and the left and right long edges.
\begin{figure}[htb]
	\centering
	\includegraphics[page=4]{proofIDEA3n}
	\caption{A 3-circle drawing of $K_{4,4,4}$ with 175 crossings.}
	\label{fig:K444175}
\end{figure}

For $n\geq 5$, we use the ideas of \cref{rem:improvedUpperBound} to improve the upper bounds.

\clearpage
\section{Connection to the Harary-Hill conjecture}\label{sec:HH}
The Harary-Hill Conjecture \cite{BW,HH} states that the number of crossings in any drawing (in the plane) of the complete graph $K_n$ is at least 
$$H(n):=\frac{1}{4}\bigg\lfloor\frac{n}{2}\bigg\rfloor
\left\lfloor\frac{{n-1}}{2}\right\rfloor
\left\lfloor\frac{{n-2}}{2}\right\rfloor
\left\lfloor\frac{{n-3}}{2}\right\rfloor.$$
Drawings with exactly $H(n)$ crossings \cite{BK,HH}
show that $\crg(K_n)\leq H(n)$. The Harary-Hill conjecture has been confirmed for  $n \leq 12$ (see \cite{G1972} for $n\leq 10$ and \cite{PR} for $n=11,12$), and either $\crg(K_{13})= H(13)$ or $H(13)-2$ \cite{AAFHPPRSV}. The conjecture has been proved when restricted to certain families of graphs \cite{AAFMMMRRV, AAFRS14,AAFRS,AAFRS13,BF,MO}.

For decades, only two families of drawings of $K_n$ with $H(n)$ crossings were known, shown in \cref{fig:optimalDrKn}: the Bla\v{z}ek-Koman construction \cite{BK}, which is an instance of a 1-circle drawing, and the Harary-Hill construction \cite{HH}, which is an instance of a balanced restricted 2-circle drawing. 
\begin{figure}[htb]
	\centering
	\begin{subfigure}[t]{.45\textwidth}
		\centering
		\includegraphics[page=2]{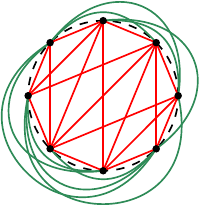}
		\caption{Construction of Bla\v{z}ek and Koman; a 1-circle drawing.}
		\label{fig:optimalDrKnA}
	\end{subfigure}
	\hfil
	\begin{subfigure}[t]{.45\textwidth}
		\centering
		\includegraphics[page=3]{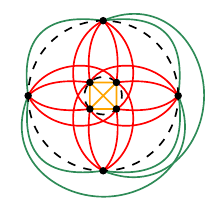}
		\caption{Construction of Harary and Hill; a  2-circle drawing.}
		\label{fig:optimalDrKnB}
	\end{subfigure}
	\caption{Drawings of $K_8$ with $H(8)=18$ crossings.
	}\label{fig:optimalDrKn}
\end{figure}

\'{A}brego et al.~announced in \cite{AAFRV} a new family of drawings of $K_n$ having the property that each edge is crossed at least once, so these are not $k$-circle drawings. Kyn\v{c}l and others \cite{JK} naturally asked about the existence of alternative $k$-circle constructions of $K_n$ with $H(n)$ crossings: Is $H(n)=\bcrN{k}(K_n)$ for some $k\geq 3$? (Recall that $\bcrN{k}$ denotes the minimum number of crossings in a balanced restricted $k$-circle crossing drawing.)
\cref{fig:3circleK6} 
\begin{figure}[htb]
	\centering
	\begin{subfigure}[t]{.31\textwidth}
		\centering
		\includegraphics[page=5]{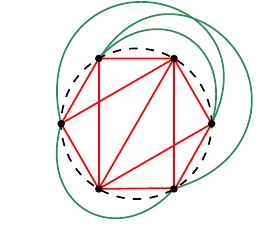}
		\caption{1-circle drawing}
		\label{fig:3circleK6A}
	\end{subfigure}
	\hfil
	\begin{subfigure}[t]{.31\textwidth}
		\centering
		\includegraphics[page=6]{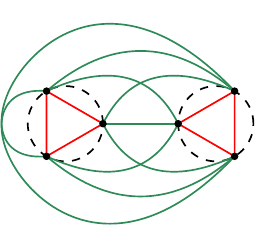}
		\caption{2-circle drawing}
		\label{fig:3circleK6B}
	\end{subfigure}
	\hfil
	\begin{subfigure}[t]{.31\textwidth}
		\centering
		\includegraphics[page=7]{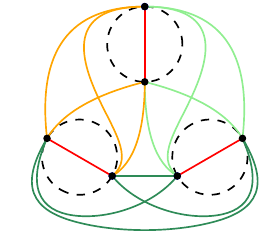}
		\caption{3-circle drawing}
		\label{fig:3circleK6C}
	\end{subfigure}
	\caption{Balanced 1-circle, 2-circle and 3-circle drawings of $K_6$, each with $3=H(6)$ crossings.}\label{fig:3circleK6}
\end{figure}
shows crossing-optimal balanced 1-circle, 2-circle, and 3-circle drawings of $K_6$. 
We prove that balanced restricted 3-circle drawings are suboptimal for $n$ large enough.

\vspace{1em}

\HH*

\begin{proof}
	Suppose $n\geq 14$. Let $q\geq 5$ and $r\in\{-1,0,1\}$ be the unique integers such that $n=3q+r$. We want to show that $\bcrN{3}(K_n)- H(n)>0$. Consider a balanced restricted 3-circle drawing of $K_n$ with $q$, $q$, and $q+r$ vertices on the 3 circles. Then
	\begin{equation}\label{eq:HH1}
		\bcrN{3}(K_n)=\crN{3}(K_{q,q,q+r})+2\binom{q}{4}+\binom{q+r}{4}.
	\end{equation}
	We use \autoref{th:general} to bound $\crN{3}(K_{q,q,q+r})$ and $\lfloor q/2 \rfloor\lfloor(q-1)/2\rfloor \geq q(q-2)/4$ to remove the floor function, which yields
	\begin{align*}
		\crN{3}(K_{q,q,q+r})
		&\geq\crN{2}(K_{q,q})+2\crN{2}(K_{q,q+r})
		+q^2\left\lfloor \frac{q+r}{2}\right\rfloor\left\lfloor \frac{q+r-1}{2}\right\rfloor
		+2q(q+r)\bigg\lfloor \frac{q}{2}\bigg\rfloor\left\lfloor \frac{q-1}{2}\right\rfloor\\
		&\geq  \crN{2}(K_{q,q})+2\crN{2}(K_{q,q+r})+\frac{1}{4}q^2(q+r)(3q+r-6).
	\end{align*}
	By \autoref{eq:bipartitebalanced}, we have $\crN{2}(K_{q,q})=q\binom{q}{3}.$ 
	By \autoref{eq:bipartiteGeneral}, we obtain $\crN{2}(K_{q,q-1})=(q-2)\binom{q}{3}$ and $\crN{2}(K_{q,q+1})=(q-1)\binom{q+1}{3}$. Thus
	\begin{equation}\label{eq:HH2}
		\crN{3}(K_{q,q,q+r})
		\geq
		\begin{cases} 
			3q\binom{q}{3}+\frac{1}{4}q^3(3q-6)\smallskip &\mbox{if $r=0$,} \\
			(3q-4)\binom{q}{3}+\frac{1}{4}q^2(q-1)(3q-7)&\mbox{if $r=-1$,}\\
			q\binom{q}{3}+2(q-1)\binom{q+1}{3}+\frac{1}{4}q^2(q+1)(3q-5)&\mbox{if $r=1$.}
		\end{cases}
	\end{equation}
	The result holds for $n\geq 14$ ($q\geq 5$) by \eqref{eq:HH1}, \eqref{eq:HH2}, and $H(n)\leq \frac{1}{64}(n-1)^2(n-3)^2$:
	\begin{equation*}
		\bcrN{3}(K_n)- H(n)\geq \begin{cases} 
			\frac{1}{64}(7q^4-24q^3-46q^2+24q-9)>0\smallskip &\mbox{if $r=0$,} \\
			\frac{q}{192}(21q^3-100q^2-36q+112)>0\smallskip &\mbox{if $r=-1$,}\\
			\frac{q}{192}(q+2)(21q^2-86q-8)>0 &\mbox{if $r=1$.}
		\end{cases}
	\end{equation*}
	Finally,  by \eqref{eq:HH1} and \cref{cor:improvedLower}, $\bcrN{3}(K_{9})=\crN{3}(K_{3,3,3})\geq 38>36=H(9)$, $\bcrN{3}(K_{10}) = \crN{3}(K_{3,3,4}) + 1 \geq 64 > 60 =H(10)$, and $\bcrN{3}(K_{13})= \crN{3}(K_{4,4,5}) + 7 \geq 229 > 225 =H(13)$. 
\end{proof}

Our previous argument does not settle the cases $n=8$, $n=11$, and $n=12$. The Harary-Hill constructions for $n\leq 5$ are in fact balanced restricted 3-circle drawings, and we give balanced restricted 3-circle drawings of $K_6$ and $K_7$ {in \cref{fig:3circleK6C,fig:k223}, respectively}.
If we allow for unbalanced constructions for $n \geq 8$, {the crossing number of $\crN{3}(K_{2,2,n-4})$ given in \cref{eq:K22n}}  implies that 
$$\crN{3}(K_{2,2,n-4})+\binom{n-4}{4}=\binom{n-4}{4} + \frac{3}2(n-4)^2-(n-4)-\begin{cases}3& n \text{ even}\\3/2&n \text{ odd}\end{cases}\geq H(n),$$
with equality if and only if $8\leq n \leq 11$. That is, for $n \geq 8$ the drawings of $K_n$ given by a crossing-optimal 3-circle drawing of $K_{2,2,n-4}$ together with the straight-line drawings inside of the three circles achieve $H(n)$ crossings if and only if $8\leq n\leq 11$.

\begin{figure}[htb]
	\centering
	\includegraphics[page=2]{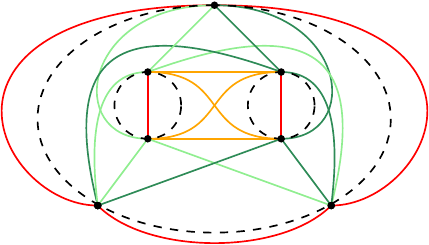}
	\caption{Balanced 3-circle drawing of $K_7$ with $9=H(7)$ crossings.}
	\label{fig:k223}
\end{figure}


\section{Conclusion and open problems}\label{sec:open}

In this paper, we prove upper and lower bounds on the tripartite-circle crossing number of complete tripartite graphs. For the lower bound, we introduce formulas describing the number of crossings in a tripartite-circle drawing. For the upper bounds, we present drawings. 
While there exist restricted balanced $2$-circle drawings achieving the Harary-Hill bound, our results imply that this is not the case for balanced restricted $3$-circle drawings for $n \ge 13$. It remains open for future work whether the same holds for $k$-partite circle drawings when $k > 3$. 
We have made progress in the direction of extending our work to $k > 3$ and plan to return to this question in a subsequent paper. We conclude with a list of interesting open problems for future work: 
\begin{itemize}
	\item Do there exist $k$-circle drawings achieving the Harary-Hill bound for $k>3$?
	\item Can the number of crossings of $k$-circle drawings generally be described by labels analogous to $x$-,$y$-labels?
	\item How are crossing-minimal $k$-circle drawings characterized? 
	\item What are the exact values for small graphs? Is $\crN{3}(K_{3,3,3})=42$? For the remaining displayed values in \cref{tab:smallN}, we believe that the truth lies closer to the presented upper bounds. In particular, it remains to develop better tools in order to improve the lower bounds.
	\item Are there balanced restricted $3$-circle drawings achieving the Harary-Hill bound for $K_8$, $K_{11}$, or $K_{12}$?
	\item We know that balanced restricted $3$-circle drawings do not (in general) achieve the Harary-Hill bound, and that extending tripartite-circle drawings of $K_{2,2,n-4}$ does not (in general) achieve the Harary-Hill bound. Are there other unbalanced restricted $3$-circle drawings that do achieve the Harary-Hill bound? 
\end{itemize}

\section*{Acknowledgments}
As this project started at the MRC workshop 
``Beyond Planarity: Crossing Numbers of Graphs'', we thank the organizers and all participants for the fruitful atmosphere.
We thank Julia B{\"o}ttcher for helpful feedback on a draft of this manuscript. Silvia Fern\'andez-Merchant was supported by the NSF grant DMS 1400653.  Marija Jeli{\'c} Milutinovi{\'c} was supported by Ministry of Education, Science and Technological Development of Serbia, Grant 174034.

	\selectlanguage{english}
	
	\bibliographystyle{plainurl}
	
	\bibliography{bib}

\begin{thebibliography}{10}

\bibitem{AAFHPPRSV}
Bernardo~M {\'A}brego, Oswin Aichholzer, Silvia Fern{\'a}ndez-Merchant, Thomas
  Hackl, J{\"u}rgen Pammer, Alexander Pilz, Pedro Ramos, Gelasio Salazar, and
  Birgit Vogtenhuber.
\newblock All good drawings of small complete graphs.
\newblock In {\em Proc. 31st European Workshop on Computational Geometry
  (EuroCG)}, pages 57--60, 2015.

\bibitem{AAFMMMRRV}
Bernardo~M. {\'A}brego, Oswin Aichholzer, Silvia Fern{\'a}ndez-Merchant, Dan
  McQuillan, Bojan Mohar, Petra Mutzel, Pedro Ramos, R.~Bruce Richter, and
  Birgit Vogtenhuber.
\newblock Bishellable drawings of {$K_n$}.
\newblock {\em SIAM Journal on Discrete Mathematics}, 32(4):2482--2492, 2018.
\newblock \href {https://doi.org/10.1137/17M1147974}
  {\path{doi:10.1137/17M1147974}}.

\bibitem{AAFRS13}
Bernardo~M. {\'A}brego, Oswin Aichholzer, Silvia Fern{\'a}ndez-Merchant, Pedro
  Ramos, and Gelasio Salazar.
\newblock The 2-page crossing number of {$K_n$}.
\newblock {\em Discrete \& Computational Geometry}, 49(4):747--777, 2013.
\newblock \href {https://doi.org/10.1007/s00454-013-9514-0}
  {\path{doi:10.1007/s00454-013-9514-0}}.

\bibitem{AAFRS14}
Bernardo~M {\'A}brego, Oswin Aichholzer, Silvia Fern{\'a}ndez-Merchant, Pedro
  Ramos, and Gelasio Salazar.
\newblock More on the crossing number of {$K_n$} : Monotone drawings.
\newblock {\em Electronic Notes in Discrete Mathematics}, 44:411--414, 2013.
\newblock \href {https://doi.org/10.1016/j.endm.2013.10.064}
  {\path{doi:10.1016/j.endm.2013.10.064}}.

\bibitem{AAFRS}
Bernardo~M. {\'A}brego, Oswin Aichholzer, Silvia Fern{\'a}ndez-Merchant, Pedro
  Ramos, and Gelasio Salazar.
\newblock Shellable drawings and the cylindrical crossing number of {$K_n$}.
\newblock {\em Discrete \& Computational Geometry}, 52(4):743--753, 2014.
\newblock \href {https://doi.org/10.1007/s00454-014-9635-0}
  {\path{doi:10.1007/s00454-014-9635-0}}.

\bibitem{AAFRV}
Bernardo~M. {\'A}brego, Oswin Aichholzer, Silvia Fern{\'a}ndez-Merchant, Pedro
  Ramos, and Birgit Vogtenhuber.
\newblock Non-shellable drawings of {$K_n$} with few crossings.
\newblock In {\em Proc. 26th Annual Canadian Conference on Comp. Geom. (CCCG)},
  2014.
\newblock URL:
  \url{http://www.csun.edu/~SF70713/publications/NewFamilyCCCG2014.pdf}.

\bibitem{AFS}
Bernardo~M. {\'A}brego, Silvia Fern{\'a}ndez-Merchant, , and Athena Sparks.
\newblock The cylindrical crossing number of the complete bipartite graph.
\newblock {\em Graphs and Combinatorics}, 36:205–220, 2019.
\newblock \href {https://doi.org/10.1007/s00373-019-02076-5}
  {\path{doi:10.1007/s00373-019-02076-5}}.

\bibitem{A1986}
Kouhei Asano.
\newblock The crossing number of {$K_{1,3,n}$ and $K_{2,3,n}$}.
\newblock {\em Journal of Graph Theory}, 10(1):1--8, 1986.
\newblock \href {https://doi.org/10.1002/jgt.3190100102}
  {\path{doi:10.1002/jgt.3190100102}}.

\bibitem{B2007}
Christian Bachmaier.
\newblock A radial adaptation of the sugiyama framework for visualizing
  hierarchical information.
\newblock {\em IEEE Transactions on Visualization and Computer Graphics},
  13(3):583--594, 2007.
\newblock \href {https://doi.org/10.1109/TVCG.2007.1000.}
  {\path{doi:10.1109/TVCG.2007.1000.}}

\bibitem{BF}
Martin Balko, Radoslav Fulek, and Jan Kyn{\v{c}}l.
\newblock Crossing numbers and combinatorial characterization of monotone
  drawings of {$K_n$}.
\newblock {\em Discrete \& Computational Geometry}, 53(1):107--143, 2015.
\newblock \href {https://doi.org/10.1007/s00454-014-9644-z}
  {\path{doi:10.1007/s00454-014-9644-z}}.

\bibitem{BW}
Lowell Beineke and Robin Wilson.
\newblock The early history of the brick factory problem.
\newblock {\em The Mathematical Intelligencer}, 32(2):41--48, 2010.
\newblock \href {https://doi.org/10.1007/s00283-009-9120-4}
  {\path{doi:10.1007/s00283-009-9120-4}}.

\bibitem{BK}
Jaroslav Bla{\v z}ek and Milan Koman.
\newblock A minimal problem concerning complete plane graphs.
\newblock {\em Theory of graphs and its applications, Czech. Acad. of Sci},
  pages 113--117, 1964.

\bibitem{K22n}
Charles Camacho, Silvia Fern{\'a}ndez-Merchant, Rachel Kirsch, Linda Kleist,
  Elizabeth Matson, Marija~Jeli{\'c} Milutinovi{\'c}, and Jennifer White.
\newblock The tripartite-circle crossing number of graphs with two small
  partition classes.
\newblock arxiv preprint: 2108.01032, 2021.
\newblock URL: \url{https://arxiv.org/abs/2108.01032}.

\bibitem{FGHLM}
Frank Duque, Hern{\'a}n Gonz{\'a}lez-Aguilar, C{\'e}sar
  Hern{\'a}ndez-V{\'e}lez, Jes{\'u}s Lea{\~n}os, and Carolina Medina.
\newblock The complexity of computing the cylindrical and the $ t $-circle
  crossing number of a graph.
\newblock {\em The Electronic Journal of Combinatorics}, 25(2):P2--43, 2018.
\newblock \href {https://doi.org/10.37236/7193} {\path{doi:10.37236/7193}}.

\bibitem{garey1983NP}
Michael~R. Garey and David~S. Johnson.
\newblock Crossing number is {NP}-complete.
\newblock {\em SIAM Journal on Algebraic Discrete Methods}, 4(3):312--316,
  1983.
\newblock \href {https://doi.org/10.1137/060403} {\path{doi:10.1137/060403}}.

\bibitem{GHLPRY}
Ellen Gethner, Leslie Hogben, Bernard Lidick{\`y}, Florian Pfender, Amanda
  Ruiz, and Michael Young.
\newblock On crossing numbers of complete tripartite and balanced complete
  multipartite graphs.
\newblock {\em Journal of Graph Theory}, 84(4):552--565, 2017.
\newblock \href {https://doi.org/10.1002/jgt.22041}
  {\path{doi:10.1002/jgt.22041}}.

\bibitem{GM}
M.~Ginn and F.~Miller.
\newblock The crossing number of {$K_{3,3,n}$}.
\newblock {\em Congr. Numer.}, 221:49--54, 2014.

\bibitem{G1960}
Richard~K. Guy.
\newblock A combinatorial problem.
\newblock {\em Nabla (Bulletin of the Malayan Mathematical Society)}, 7:68--72,
  1960.

\bibitem{G1972}
Richard~K. Guy.
\newblock Crossing numbers of graphs.
\newblock In {\em Graph Theory and Applications}, pages 111--124. Springer,
  1972.

\bibitem{HH}
Frank Harary and Anthony Hill.
\newblock On the number of crossings in a complete graph.
\newblock {\em Proceedings of the Edinburgh Mathematical Society},
  13(4):333--338, 1963.
\newblock \href {https://doi.org/10.1017/S0013091500025645}
  {\path{doi:10.1017/S0013091500025645}}.

\bibitem{HlinenyNP}
Petr Hlin{\v{e}}n{\`y}.
\newblock Crossing number is hard for cubic graphs.
\newblock {\em Journal of Combinatorial Theory, Series B}, 96(4):455--471,
  2006.
\newblock \href {https://doi.org/10.1016/j.jctb.2005.09.009}
  {\path{doi:10.1016/j.jctb.2005.09.009}}.

\bibitem{JK}
Jan Kyn\v{c}l.
\newblock Drawings of complete graphs with {Z(n)} crossings.
\newblock URL: \url{https://mathoverflow.net/questions/128878}.

\bibitem{leighton}
Frank~Thomson Leighton.
\newblock {\em Complexity Issues in VLSI}.
\newblock MIT press, 1983.
\newblock Available at
  \url{https://mitpress.mit.edu/books/complexity-issues-vlsi}.

\bibitem{MO}
Petra Mutzel and Lutz Oettershagen.
\newblock The crossing number of semi-pair-shellable drawings of complete
  graphs.
\newblock In {\em Proceedings of the 30th Canadian Conference on Computational
  Geometry (CCCG)}, pages 11--17, 2018.
\newblock URL:
  \url{http://www.cs.umanitoba.ca/\%7Ecccg2018/papers/session1A-p2.pdf}.

\bibitem{PR}
Shengjun Pan and R.~Bruce Richter.
\newblock The crossing number of {$K_{11}$} is 100.
\newblock {\em Journal of Graph Theory}, 56(2):128--134, 2007.
\newblock \href {https://doi.org/10.1002/jgt.20249}
  {\path{doi:10.1002/jgt.20249}}.

\bibitem{PSS}
Michael~J. Pelsmajer, Marcus Schaefer, and Daniel {\v{S}}tefankovic.
\newblock Odd crossing number and crossing number are not the same.
\newblock {\em Discrete and Computational Geometry}, 39:442--454, 2008.
\newblock \href {https://doi.org/10.1007/978-0-387-87363-3_23}
  {\path{doi:10.1007/978-0-387-87363-3_23}}.

\bibitem{RT}
R.~Bruce Richter and Carsten Thomassen.
\newblock Relations between crossing numbers of complete and complete bipartite
  graphs.
\newblock {\em The American Mathematical Monthly}, 104(2):131--137, 1997.
\newblock \href {https://doi.org/10.1080/00029890.1997.11990611}
  {\path{doi:10.1080/00029890.1997.11990611}}.

\bibitem{Sch14}
Marcus Schaefer.
\newblock The graph crossing number and its variants: A survey.
\newblock {\em The Electronic Journal of Combinatorics}, page DS21, 2021.
\newblock \href {https://doi.org/10.37236/2713} {\path{doi:10.37236/2713}}.

\bibitem{Z}
Kazimierz Zarankiewicz.
\newblock On a problem of {P.} {T}ur{\'a}n concerning graphs.
\newblock {\em Fundamenta Mathematicae}, 1(41):137--145, 1955.
\newblock URL: \url{http://matwbn.icm.edu.pl/ksiazki/fm/fm41/fm41117.pdf}.

\end{thebibliography}
	
\end{document}